\documentclass[notitlepage]{article}
\usepackage[left=1in, right=1in, top=1in, bottom=1in]{geometry}

\usepackage{titling}
\usepackage{lipsum}
\usepackage{amsthm}
\usepackage{hyperref}
\usepackage{cleveref}
\usepackage{mathrsfs}
\usepackage{amssymb}
\usepackage{amsmath}

\DeclareMathAlphabet{\mathpzc}{OT1}{pzc}{m}{it}

\newcommand{\Trans}{\operatorname{Trans}}

\newcommand{\Aut}{\operatorname{Aut}}
\newcommand{\Hy}{\mathcal{H}}

\newcommand{\RR}{\mathbb{R}}
\newcommand{\NN}{\mathbb{N}}

\newcommand{\SL}{\operatorname{SL}}

\newcommand{\PSL}{\operatorname{PSL}}
\newcommand{\SO}{\operatorname{SO}}

\newcommand{\colim}{\operatorname{colim}}
\newcommand{\Stab}{\operatorname{Stab}}
\newcommand{\s}{\mathpzc{s}}
\newcommand{\tr}{\mathpzc{t}}
\newcommand\restr[2]{{
  \left.\kern-\nulldelimiterspace 
  #1 
  \vphantom{\big|} 
  \right|_{#2} 
  }}

\newtheorem{theorem}{Theorem}
\newtheorem{lemma}[theorem]{Lemma}
\newtheorem{proposition}[theorem]{Proposition}
\theoremstyle{definition}
\newtheorem{definition}[theorem]{Definition}

\pretitle{\begin{center}\Large\bfseries}
\posttitle{\par\end{center}\vskip 0.5em}
\preauthor{\begin{center}\Large\ttfamily}
\postauthor{\end{center}}
\predate{\par\large\centering}
\postdate{\par}

\title{An extension theorem for embedded Riemannian symmetric spaces of non-compact type and an application to their universal property}
\author{Julius Gr\"uning, Ralf K\"ohl}
\date{\today}
\begin{document}

\maketitle
\thispagestyle{empty}

\begin{abstract}
By \cite{FHHK} it is known that a geodesic $\gamma$ in an abstract reflection space $X$ (in the sense of Loos, without any assumption of differential structure) canonically admits an action of a $1$-parameter subgroup of the group of transvections of $X$. In this article, we modify these arguments in order to prove an analog of this result stating that, if $X$ contains an embedded hyperbolic plane $\Hy\subset X$, then this yields a canonical action of a subgroup of the transvection group of $X$ isomorphic to a perfect central extension of $\PSL_2(\RR)$. This result can be further extended to arbitrary Riemannian symmetric spaces of non-compact type $Y$ lying in $X$ and can be used to prove that a Riemannian symmetric space and, more generally, the Kac--Moody symmetric space $G/K$ for an algebraically simply connected two-spherical Kac--Moody group $G$, as defined in \cite{FHHK}, satisfies a universal property similar to the universal property that the group $G$ satisfies itself.
\end{abstract}

\section{Introduction}
In \cite{FHHK} a \emph{symmetric space} is defined (following Ottmar Loos \cite{Loos}) in a very general fashion, allowing for many different contexts, and is then applied to the Kac--Moody setting. The definition given -- and the one which is also used in this article -- is the following:
\begin{definition}
Let $X$ be a topological space and $\mu:X\times X\to X$, $(x,y)\mapsto x.y$ a continuous map satisfying the following axioms:
$$\begin{array}{lrl}
\text{(RS1)}& \forall x\in X:&x.x=x \\
\text{(RS2)}& \forall x,y\in X:& x.(x.y)=y\\
\text{(RS3)}& \forall x,y,z\in X:& x.(y.z)=(x.y).(x.z)\\
\text{(RS4)}& \forall x,y\in X:& x.y=y \Rightarrow x=y
\end{array}$$
Then $X$ is called a \emph{(topological) symmetric space} and $\mu$ is called its \emph{reflection map}. A space satisfying only (RS1)--(RS3) is called a \emph{(topological) reflection space}.
\end{definition}

A \emph{morphism} between reflection spaces $X\to Y$ is a continuous map $\phi:X\to Y$ satisfying $$\phi(x.y)=\phi(x).\phi(y)$$ for all $x,y\in X$. In particular, for each $x \in X$ the \emph{elementary reflection} $$\begin{array}{rcl}\s_x:X&\to& X\\ y&\mapsto& x.y\end{array}$$ provides a morphism of $X$, in fact an automorphism.
Note that a morphism $\phi : X \to Y$ of reflection spaces allows one to transport the elementary reflections of $X$ to elementary reflections of $Y$: $$\s_{\phi(x)}(\phi(y)) = \phi(x).\phi(y) = \phi(x.y) = \phi(\s_x(y)).$$ That is, it is meaningful to denote $\phi(\s_x) := \s_{\phi(x)}$. 

This article is mainly concerned with the category of pointed reflection spaces: every reflection space is assumed to have a chosen base point, which is most of the time called $o$ -- and morphisms other than reflections are usually required to map basepoints to basepoints.

\smallskip
The subgroup of $\mathrm{Aut}(X)$ generated by the elementary reflections is called the \emph{main group} $G(X)$ of $X$. We stress that the assignment $\s_x\mapsto\s_{\phi(x)}$ induced by a morphism $\phi : (X,o) \to (Y,o)$ as discussed above does not necessarily induce a group homomorphism $G(X)\to G(Y)$.

The main group has a subgroup of index $\leq 2$ called the \emph{group of transvections} of $X$, which is generated by products of pairs of elementary reflections:
$$\Trans(X):=\langle \s_x\circ s_y\mid x,y\in X\rangle.$$
Given a basepoint $o\in X$, one may also define \emph{elementary transvections}
$$\tr_x:=\s_x\circ \s_o,$$ and one easily sees that these also generate $\Trans(X)$.
Given an embedding $\iota : (Y,o) \hookrightarrow (X,o)$ of pointed reflection spaces, we can consider the \emph{restricted group of transvections} of $X$ to $Y$, which is a subgroup of $\Trans(X)$: $$\restr{\Trans(X)}{Y}:=\langle \tr_{\iota(y)} \in\Trans(X)\mid y \in Y\rangle = \langle \s_{\iota(y)} \circ \s_o \in\Trans(X) \mid y \in Y \rangle.$$


With this wording, proposition 2.19 of \cite{FHHK} then implies, that if $\gamma\subset X$ is a geodesic (that is, an embedded reflection space $(\mathbb{R}, x.y := 2x-y)$), then $$\restr{\Trans(X)}{\gamma}\cong (\RR,+).$$ In \cite{Neeb}, theorem 1.6, a variation of this result is stated and proved that also encompasses not necessarily injective morphisms $(\mathbb{R}, x.y := 2x-y) \to X$.

\medskip

If instead of a geodesic one considers an embedded Riemannian symmetric space $Y \hookrightarrow X$, then unlike for geodesics one cannot expect $\Trans(Y)$ to embed into $\Trans(X)$. Indeed, for instance the element $$\begin{pmatrix} -1 & 0 & 0 \\ 0 & -1 & 0 \\ 0 & 0 & 1 \end{pmatrix} \in \mathrm{SL}_3(\mathbb{R})$$ acts non-trivially on the symmetric space $\mathrm{SL}_3(\mathbb{R})/\mathrm{SO}(3)$, but trivially on the hyperbolic plane embedded into the upper left corner.

This problem arises from the fact that centres of nested semisimple Lie groups generally do not nest. 
Our first theorem states that this in fact is the only obstruction.
%
\begin{theorem}\label{th1} \label{th2}
Let $X$ be a reflection space and $Y\subset X$ an embedded Riemannian symmetric space of non-compact type. Then
$\restr{\Trans(X)}{Y}$ is a perfect central extension of $\Trans(Y)$.
\end{theorem}

This extension theorem yields a result about a universal property that Riemannian symmetric spaces of non-compact type and, more generally, the unreduced Kac--Moody symmetric space of a two-spherical Kac--Moody group satisfies.

Recall that any algebraically simply connected semisimple split real Lie group or any algebraically simply connected two-spherical split real Kac--Moody group $G$ may be defined as the colimit (in the category of Hausdorff topological groups) of a diagram $\delta$ given by copies of $\SL_2(\RR)$ and algebraically simply connected rank-$2$ split real Lie groups, cf.~Theorem~7.22 of \cite{HKM}.

By choosing a Cartan--Chevalley involution of $G$ that normalizes the embedded rank-$1$ and rank-$2$ subgroups, this yields a diagram $\delta'$ in the category of pointed topological symmetric spaces consisting of copies of hyperbolic planes and rank-$2$ Riemannian symmetric spaces of non-compact type.

Our second main result confirms that the colimit of $\delta'$ exists and coincides with the Riemannian resp.~ Kac--Moody symmetric space $G/K$.

\begin{theorem}\label{th3}
Let $G$ be an irreducible algebraically simply connected semisimple split real Lie group or an irreducible algebraically simply connected two-spherical split real Kac--Moody group, $\theta$ its Cartan--Chevalley involution, $K:=G^\theta$ the fixed point set of $\theta$, and $\delta'$ the canonical diagram of pointed Riemannian symmetric spaces of ranks one and two embedded into $(G/K,o=K)$. Then $$(G/K,o) \cong\colim(\delta').$$
\end{theorem}
\vspace{1cm}
\textbf{Acknowledgments:} The authors express their gratitude to the Deutsche Forschungsgemeinschaft for funding the research leading to this article via the project KO4323/13-1.

\section{Semisimple Lie groups and Kac--Moody groups}
As noted in the introduction, an algebraically simply connected $2$-spherical split real Kac--Moody group may be defined as a certain colimit. This approach is due to Abramenko--M\"uhlherr \cite{AM} as a generalization of Tits original observation \cite{Tits:1974} concerning semisimple split algebraic groups/Lie groups, and has been topologized as Theorem 7.22 of \cite{HKM}.

\begin{definition}
Starting with a two-spherical Dynkin diagram $\Gamma$, for each vertex $i$ of $\Gamma$ we define $$G_i:=\SL_2(\RR).$$ For each pair of vertices $\{i,j\}$ with $i\not=j$ we define $G_{\{i,j\}}$ to be the algebraically simply connected split real Lie group of rank two whose Dynkin diagram is the full labelled subgraph of $\Gamma$ on the vertices $i$ and $j$.
A fixed choice of a root basis provides natural inclusion maps $$G_i\hookrightarrow G_{\{i,j\}}.$$ Now let $\delta$ denote the diagram given by the $G_i$ and $G_{\{i,j\}}$ and these inclusion maps. The  \emph{Kac--Moody group} $G$ is now defined as the colimit of $\delta$ in the category of topological groups and, by \cite{HKM}, turns out to be a Hausdorff topological group.

In case $G$ admits a finite Weyl group it is abstractly isomorphic to an algebraically simply connected semisimple split real Lie group and the thus defined topology coincides with its Lie group topology by an open-mapping theorem, cf.\ Proposition~2.2 of \cite{GGH} and Corollary 7.16 of \cite{HKM}. 

Each of the standard rank one subgroups $G_i\cong\SL_2(\RR)$ admits a continuous involution $\theta_i$ induced by $g\mapsto (g^T)^{-1}$. By universality these involutions $\theta_i$ extend uniquely to an involution $$\theta:G\to G,$$ called the \emph{Cartan--Chevalley involution} of $G$. We denote the set of fixed points, which is a subgroup of $G$, by $$K:=G^\theta=\{k\in G\mid \theta(k)=k\}.$$
\end{definition}

\section{Reflection Spaces and Symmetric Spaces}
We now recall some of the facts about symmetric spaces from \cite{FHHK} that are needed for the understanding of our theorems and their proofs.
One of the key ingredients is the following statement, which we will refer to as the \emph{conjugation formula}. It is lemma 2.4 in \cite{FHHK}; see also \cite{Caprace}.

\begin{lemma}\label{conjugateformula}
Let $X$ be a reflection space and $x,y\in X$. Let $\alpha\in\Aut(X)$ be an automorphism of $X$. Then $${}^\alpha(\s_y)=\s_{\alpha(y)}.$$
In particular $${}^{\s_x}(\s_y)=\s_{x.y}.$$
Here ${}^hg$ denotes the left-conjugation of $g$ by $h$.
\end{lemma}

This lemma makes a very strong statement about the relations that hold in the groups $G(X)$ or $\Trans(X)$. In a sense all the computations done in this article can be carried out in the group that is freely generated by the symbols $\s_x$, $x\in X$, modulo the relators $\s_x^2$ and ${}^{\s_x}(\s_y){\s_{x.y}}^{-1}$.\\

A first application of the conjugation formula is the statement that for an embedded reflection space $\iota : (Y,o) \hookrightarrow (X,o)$ of $X$ the natural map $\restr{\Trans(X)}{Y}\to \Trans(Y)$ induced by $\s_{\iota(y)} \circ \s_o \mapsto \s_y \circ \s_o$, $y \in Y$, has central kernel:
\begin{lemma}\label{centralextensionlemma}
Let $Y\subset X$ be a subspace of the reflection space $X$. Then the natural map $\restr{\Trans(X)}{Y}\to \Trans(Y)$ has central kernel.
\end{lemma}
\begin{proof}
 An element $g\in\restr{\Trans(X)}{Y}$ in the kernel of that map acts trivially on the subspace $Y\subset X$. So by the conjugation formula we get $${}^g(\s_y)=\s_{g(y)}=\s_y,$$ for all $y\in Y$, so $g$ centralizes the elements $\s_y \circ s_o$ for all $y \in Y$ and, hence, lies in the center of $\restr{\Trans(X)}{Y}$.\end{proof}

We conclude that for the proof of \cref{th1} we need only show that the respective groups are perfect. 

\smallskip
Another application of the conjugation formula is the following lemma, which makes use of axiom (RS4) and so only applies to symmetric spaces:
\begin{lemma}\label{centraliserprop} Let $X$ be a symmetric space and let $o\in X$. Let $H\leq \Aut(X)$ some subgroup. Then $$\Stab_H(o)=C_H(\s_o).$$\end{lemma}
\begin{proof}
For any $g\in\Stab_H(o)$, we have that $g(o)=o$. This means with the conjugation formula that $${}^g(\s_o)=\s_{g(o)}=\s_o,$$ so $g\in C_H(\s_o)$. Now if $g\in C_H(\s_o)$, we have $\s_{g(o)}={}^g(\s_o)=\s_o$ and hence $$g(o).o=\s_{g(o)}(o)=\s_o(o)=o.o=o,$$ and with (RS4) it follows $g(o)=o$, so that $g\in\Stab_H(o)$.
\end{proof}

For the understanding of the construction of a topological reflection space from a topological group $G$ together with a continuous involution $\theta\in\Aut(G)$, we want to recall proposition 4.1 of \cite{FHHK}, which we refer to as the \emph{$G/K$-construction} or the \emph{$G/K$-functor}:
\begin{proposition}\label{G/Kconstruction} Let $G$ be a topological group and $\theta\in\Aut(G)$ an involution. We denote $K:=G^\theta$ the fixed point set of $\theta$ and $$\begin{array}{rcl}\tau:G&\to& G\\ g&\mapsto& g\theta(g)^{-1},\end{array}$$ the so-called \emph{twist map}. Then the coset space $G/K$ becomes a topological reflection space with the reflection map 
$$\begin{array}{rcl}\mu: G/K\times G/K &\to& G/K \\ (gK,hK)&\mapsto& \tau(g)\theta(h)K.\end{array}$$
If furthermore $K\cap\tau(G)=\{e\}$, then $X$ is also a symmetric space.
\end{proposition}

Now if we have two groups $G$ and $H$ and involutions $\theta_G\in\Aut(G)$ and $\theta_H\in\Aut(H)$ satisfying the conditions of \cref{G/Kconstruction}, and a group homomorphism $\varphi:G\to H$ respecting the involutions, $$\phi\circ\theta_G=\theta_H\circ \phi,$$ then there is an induced map $G/K_G\to H/K_H$, which is a morphism of symmetric spaces. So the $G/K$-construction can be thought of as a functor from the category of topological groups with involutions to the category of pointed topological reflection spaces, the basepoint being the coset $o=K=eK\in G/K$.

This readily applies to the situation of the defining diagram $\delta$ of a Kac--Moody group $G$, yielding a diagram $\delta'$ of topological symmetric spaces.

\section{An extension theorem for embedded Riemannian symmetric spaces}

As pointed out in the section above (\cref{centralextensionlemma}), for proving \cref{th1}, we need only show that the groups in question are perfect. To accomplish this, we intend to identify a generating system such that taking suitable commutators reproduces this generating system. The proof is divided into two steps: First we are going to prove the theorem for the special case that the embedded Riemannian symmetric space is the hyperbolic plane, and then we extend this to the general case.
The following proposition produces generating systems for the groups we are looking at:

\begin{proposition}\label{generatingprop}
Let $G$, $K$ be as in the situation of \cref{G/Kconstruction}. Moreover, let $(X,o)\supset (G/K,eK)$ be a pointed reflection space containing $G/K$.
Let $\{g_i, i\in I\}$ be a generating system of the group $G$. Furthermore, let $$T_i\in\restr{\Trans(X)}{G/K}, i\in I$$ be elements that act on $G/K$ via left multiplication with $g_i$:$$T_i( hK)=(g_ih)K,$$ for all $i\in I$ and $h\in G$. Then the set $\{T_i, {}^{\s_o}(T_i^{-1}) \}$ is a generating system of the group $\restr{\Trans(X)}{G/K}$.
\end{proposition}
\begin{proof} Since the $g_i$ are a generating system of $G$, any $h$ can be written as a product $$h=g_{i_1}^{\epsilon_1}\cdots g_{i_n}^{\epsilon_n}.$$ Now let us define $$T:=T_{i_1}^{\epsilon_1}\cdots T_{i_n}^{\epsilon_n}.$$
We see with the conjugation formula, that for the elementary transvection defined by $hK$, we have $$\tr_{hK}=\s_{hK}\circ \s_o=T\s_oT^{-1}\s_o\in\langle T_i, {}^{\s_o}(T_i^{-1}), i\in I\rangle. \qedhere $$
\end{proof}

Finding elements like the $T_i$ in the preceding proposition is not necessarily an easy task. But in the case of the hyperbolic plane $G/K\cong\Hy$ with $G=\SL_2(\RR)$, $K=\SO(2)$ this is possible; the $g_i\in \SL_2(\RR)$ we are aiming for are the shear matrices $$\begin{pmatrix}1&t\\0&1\end{pmatrix},\begin{pmatrix}1&0\\t&1\end{pmatrix}, t\in\RR\setminus\{0\},$$ which are well-known to generate $\SL_2(\RR)$.
An elementary transvection $\tr_{gK}$ acts on an element $hK$ in the following way: $$\tr_{gK}(hK)=\s_{gK}(\s_{eK}(hK))=\s_{gK}(\tau(e)\theta(h)K)=(\tau(g)h)K.$$
In the case of $G=\SL_2(\RR)$ the map $\tau$ sends $g$ to $gg^T$ and in the case that $g$ is a symmetric matrix this is just $g^2$. So in order to find the $T_i$ we wish to express the shear matrices as products of positively definite symmetric matrices, as these allow for taking square roots. By \cite{Ballantine} every unipotent matrix in $\SL_2(\RR)$ is a product of at most three symmetric positive definite matrices; the following lemma gives explicit such symmetric positive matrices.

\begin{lemma}\label{matrixlemma}
The matrices $$\begin{array}{rcl}
A(t)&:=&\frac{1}{\sqrt 2}\begin{pmatrix}3t&-1\\-1&\frac{1}{t}\end{pmatrix},\\
B(t)&:=&\begin{pmatrix}\frac{1}{t}&1\\1&2t\end{pmatrix},\\
C&:=&\begin{pmatrix}\frac{1}{\sqrt2}&0\\0&\sqrt2\end{pmatrix}
\end{array}$$ are symmetric, positive definite and of determinant $1$ for all $t\in\RR\setminus\{0\}$. Furthermore $$\begin{array}{rcl}
A(t)B(t)C&=&\begin{pmatrix}1&t\\0&1\end{pmatrix},\\
CB(t)A(t)&=&\begin{pmatrix}1&0\\t&1\end{pmatrix},\\
CA(t)C&=&A(t/2), \\
C^{-1}B(t)C^{-1}&=&B(t/2).\end{array}$$
\end{lemma}

\begin{proof}
This is a straightforward calculation.
\end{proof}

We are going to use the last two equations of the lemma for a computation of relations in the group $\restr{\Trans(X)}{\Hy}$ that will allow us to establish its perfectness.

The next lemma produces two sets of generators for the group $\restr{\Trans(X)}{\Hy}$ acting via left multiplication with shear matrices on an embedded hyperbolic place $\Hy=\SL_2(\RR)/\SO(2) \hookrightarrow X$.
\begin{lemma}\label{generatorslemma} We denote by $a(t)$, $b(t)$ and $c$ the respective square roots of the matrices $A(t)$, $B(t)$ and $C$ of \cref{matrixlemma}. For brevity we write $[g]$ for the coset $g\SO(2)\in\Hy=\SL_2(\RR)/\SO(2)$. If $X\supset\Hy$ is a symmetric space containing $\Hy$, then both the sets
$$\begin{array}{ll} & \{x_+(t),x_-(t) \mid t\in\RR\setminus\{0\}\} \\
\text{and }& \{x_+(t/2)\cdot x_+(t)^{-1}, x_-(t)\cdot x_-(t/2)^{-1}\mid t\in\RR\setminus\{0\}\}, \end{array}$$
where $$\begin{array}{rcl}x_+(t)&:=&\tr_{[a(t)]}\circ \tr_{[b(t)]}\circ \tr_{[c]} \\ x_-(t)&:=&\tr_{[c]}\circ \tr_{[b(t)]}\circ \tr_{[a(t)]}\end{array}$$
are generating systems of $\restr{\Trans(X)}{\Hy}$.
\end{lemma}
\begin{proof}
As noted before, the elementary transvections $\tr_{[a(t)]}$, $\tr_{[b(t)]}$, $\tr_{[c]}$, act by left multiplication with $A(t)$, $B(t)$ and $C$ respectively. But since $$A(t)B(t)C=\begin{pmatrix}1&t\\0&1\end{pmatrix}$$ by \cref{matrixlemma}, we get that $x_+(t)$ acts by left multiplication with  $\begin{pmatrix}1&t\\0&1\end{pmatrix}$.

Similarly $x_-(t)$ acts by left multiplication with $\begin{pmatrix}1&0\\t&1\end{pmatrix}$.

Finally the elements $x_+(t/2)\cdot x_+(t)^{-1}$ and $x_-(t)\cdot x_-(t/2)^{-1}$ act by left multiplication with the matrices $$\begin{array}{rc}&\begin{pmatrix}1&-t/2\\0&1\end{pmatrix} \\ \text{resp. } & \begin{pmatrix}1&0\\t/2&1\end{pmatrix}.\end{array}$$

Now for any elementary transvection $\tr_x$, we have that ${}^{\s_o}(\tr_x^{-1})=\tr_x$, so that $$x_-(t)={}^{\s_o}(x_+(t)^{-1}),$$
which proves the statement of the lemma after applying \cref{generatingprop}.
\end{proof}

We are now in a position to prove \cref{th1}. We will start with the situation of an embedded hyperbolic plane.
\begin{proof}[Proof of \cref{th1} for embedded hyperbolic planes.]
We claim that the following equalities hold:
$$\begin{array}{rcl}
[\tr_{[c]},x_+(t)]&=&x_+(t/2)\cdot x_+(t)^{-1},\\ 
{[x_-(t),\tr_{[c]}^{-1}]}&=&x_-(t)\cdot x_-(t/2)^{-1}. 
\end{array}$$
If we can prove these equalities, we have shown, by \cref{generatorslemma}, that the commutator subgroup of $\restr{\Trans(X)}{\Hy}$ contains a generating system and hence is perfect. Moreover, the latter equality follows from the former by inverting and conjugating with $\s_o$, so it suffices to prove the former.

Since $[\tr_{[c]},x_+(t)]={}^{\tr_{[c]}}x_+(t)x_+(t)^{-1}$ it suffices to show that ${}^{\tr_{[c]}}x_+(t)=x_+(t/2).$
But $$\begin{array}{rcl}
{}^{\tr_{[c]}}x_+(t)&=&\tr_{[c]}\circ \tr_{[a(t)]}\circ \tr_{[b(t)]} \\
&=& s_{[c]}\s_os_{[a(t)]}\s_os_{[b(t)]}\s_o\\
&=& s_{[c]}s_{[\theta {a(t)}]}s_{[c]}s_{[c]}s_{[b(t)]}\s_o \\
&=& s_{[(\tau c)a(t)]}\s_o\s_os_{[c]}s_{[b(t)]}s_{[c]}s_{[c]}\s_o \\
&=& s_{[(\tau c)a(t)]}\s_o\s_os_{(\tau c)\theta b(t)}\s_o\s_os_{[c]}\s_o \\
&=& s_{[(\tau c)a(t)]}\s_os_{[(\theta\tau c)b(t)]}\s_os_{[c]}\s_o \\
&=& s_{[c^2a(t)]}\s_os_{[c^{-2}b(t)]}\s_os_{[c]}\s_o.
\end{array}$$

So it remains to show that $[c^2a(t)]=[a(t/2)]$ and $[c^{-2}b(t)]=[b(t/2)]$, i.e. that
$$\begin{array}{rl}&
a(t/2)^{-1}c^2a(t)\in\SO(2)\\ \text{and }& b(t/2)^{-1}c^{-2}b(t)\in\SO(2).\end{array}$$ But after \cref{matrixlemma} we have $$a(t/2)^{-1}c^2a(t)\cdot(a(t/2)^{-1}c^2a(t))^T=a(t/2)^{-1}c^2a(t)^2c^2a(t/2)^2=a(t/2)^{-1}a(t/2)^2a(t/2)^{-1}=1,$$ and $$b(t/2)^{-1}c^{-2}b(t)\cdot(b(t/2)^{-1}c^{-2}b(t))^T=b(t/2)^{-1}c^{-2}b(t)^2c^{-2}b(t/2)^2=b(t/2)^{-1}b(t/2)^2b(t/2)^{-1}=1. \qedhere$$
\end{proof}

The general proof of \cref{th2} now follows from classical structure results about Riemannian symmetric spaces of non-compact type and semisimple Lie groups:

\begin{proof}[Proof of \cref{th2} for embedded Riemannian symmetric spaces.]
Since $Y$ is a Riemannian symmetric space of non-compact type it follows that $Y=G/K$ for $G$ a real semisimple Lie group (see \cite{Loos}). From Steinberg (see \cite{Steinberg}) we know that $G$ is generated by subgroups $G_\alpha\cong\SL_2(\RR)$, and that $K_\alpha:=K\cap G_\alpha\cong\SO(2)$. So $\Hy_\alpha:=G_\alpha/K_\alpha$ is a hyperbolic plane. The subgroups $\restr{\Trans(X)}{\Hy_\alpha}$ are closed under conjugation with $\s_o$ by the conjugation formula, since $o=eK$ lies in every $\Hy_\alpha$. This means that $\restr{\Trans(X)}{Y}$ is generated by the subgroups $\restr{\Trans(X)}{\Hy_\alpha}$, applying \cref{generatingprop}. But the $\restr{\Trans(X)}{\Hy_\alpha}$ are perfect, as we have already proven the theorem for embedded hyperbolic planes -- so $\restr{\Trans(X)}{Y}$ is perfect, too. 
\end{proof}

\section{A universal property of Riemannian symmetric spaces and unreduced Kac--Moody symmetric spaces}
Since, in the situation of \cref{th2}, the morphism $\restr{\Trans(X)}{Y}\to\Trans(Y)$ is a perfect central extension, the universal central extension of $\Trans(Y)$ acts on $X$ by transvections. This can be used to prove that the symmetric space $G/K$ for an algebraically simply connected semisimple split real Lie group or an algebraically simply connected two-spherical split real Kac--Moody group $G$ is the colimit in the category of pointed topological symmetric spaces of the induced diagram $\delta'$ of the diagram $\delta$ of fundamental subgroups of rank-$1$ and rank-$2$ for which $G=\colim(\delta)$, after applying the $G/K$-construction.

\begin{proof}[Proof of \cref{th3}]
  Applying the $G/K$-functor to the diagram $\delta\to G$, we see that $G/K$ is a cocone for the diagram $\delta'$, so only universality needs to be checked. Let $(X,o)$ be another cocone of $\delta'$, i.e.\ there are morphisms $G_i/K_i\to X$, $G_{\{i,j\}}/K_{\{i,j\}}\to X$, making the diagram $\delta'\to X$ commute. Let $p_i : \widetilde G_i \to G_i$ and $p_{i,j} : \widetilde G_{\{i,j\}} \to G_{\{i,j\}}$ be the universal central extensions. Then $G_i/K_i \cong  \widetilde G_i/\widetilde K_i$ and $G_{\{i,j\}}/K_{\{i,j\}} \cong \widetilde G_{\{i,j\}}/ \widetilde K_{\{i,j\}}$, where $\widetilde K_i := p_i^{-1}(K_i)$ and $\widetilde K_{\{i,j\}} :=p_{i,j}^{-1}(K_{\{i,j\}})$. 
By theorem~\cref{th2} the universal central extensions $\widetilde G_i$, $\widetilde G_{\{i,j\}}$ act on $X$ through transvections. Moreover, their subgroups $\widetilde K_i$ and $\widetilde K_{\{i,j\}}$ stabilize $o$.

We conclude that the colimit of the diagram $\widetilde \delta$ of these universal central extensions, which coincides with the universal central extension $\widetilde G$ of $G$ by Theorem~B of \cite{CapraceUniv}, acts by transvections on $X$; we denote this action by $$T:\widetilde G\to \Trans(X).$$
We denote the lift of the Cartan--Chevalley involution $\theta:G\to G$ to $\widetilde G$ by $\widetilde\theta$. Certainly, $\widetilde K:=\widetilde G^{\widetilde\theta}=p^{-1}(K)$, where $p:\widetilde G\to G$ is the canonical map, and $$\widetilde G/\widetilde K\cong G/K.$$
We now define the morphism $$\begin{array}{rcl}\widetilde G/\widetilde K&\to&X\\ g\widetilde K&\mapsto& T(g)(o).\end{array}$$
This is a well defined map, since $\widetilde K$ stabilises the base point $o$:

By \cref{centraliserprop} the stabiliser $\Stab_{\Trans(X)}(o)$ coincides with the centraliser $C_{\Trans(X)}(\s_o)$, so the fixed point set of the conjugation with $\s_o$. But since $\s_o(gK)=\tau(e)\theta(g)K=\theta(gK)$, the map $G/K\to G/K$ induced by $\theta$ coincides with $\s_o$. Hence so does the map $\widetilde G/\widetilde K\to\widetilde G/\widetilde K$ induced by $\widetilde\theta$. Hence $$T(\widetilde K)=C_{\Trans(X)}(\s_o)=\Stab_{\Trans(X)}(o).$$

To prove uniqueness of the morphism $G/K\to X$, we note that $\SL_2(\RR)$ is generated by symmetric, positive definite matrices. For any such matrix $g$, we have $\tau(g)=g^2$ -- now taking the square root $h$ of $g$, we have $\tau(h)=g$. That means we can write any $x\in\SL_2(\RR)$ as a product
$$x=\tau(h_1)\cdots \tau(h_n),$$ for some $n\in\NN$. But since $G$ is generated by subgroups isomorphic to $\SL_2(\RR)$, this also holds for any $x\in G$ with the $h_j$ in suitable $G_i\cong\SL_2(\RR)$. But this shows that in the symmetric space $G/K$, we have $$xK=h_1.(o.(h_2.(o\dots h_n)\dots)),$$ i.e. that $G/K$ is generated by the subspaces $G_i/K_i$. But the image of the morphism $G/K\to X$ on these subspaces is already determined.
\end{proof}

\end{document}